\documentclass[12pt]{amsart}

\usepackage{amsmath}
\usepackage{amsfonts}
\usepackage{amssymb}
\usepackage{amsthm}
\usepackage{longtable}
\usepackage{adjustbox,lipsum}
\usepackage{graphicx}
\usepackage{comment}
\usepackage[justification=centering]{caption}
\usepackage{float}
\usepackage[mathscr]{eucal}\catcode`\@=11

\usepackage{url}
\usepackage{enumerate}

\DeclareMathOperator{\Gal}{Gal}
\DeclareMathOperator{\Frob}{Frob}
\newcommand{\cA}{\mathcal{A}}
\newcommand{\Q}{\mathbb{Q}}
\newcommand{\PP}{\mathbb{P}}

\DeclareMathOperator{\re}{Re}

\theoremstyle{plain}
\newtheorem{theorem}{Theorem}[section]
\newtheorem{defn}{Definition}[section]
\newtheorem{lemma}{Lemma}[section]

\begin{document}

\baselineskip=17pt


\title{On a character-twisted analogue of Sch\"{a}ffer's equation}
\author{Kálmán Győry, Vandita Patel, \'Akos Pint\'er, and Samir Siksek}

\address{Kálmán Győry and \'Akos Pint\'er \newline
\indent Institute of Mathematics \newline
\indent University of Debrecen \newline
\indent P.O. Box 400, H-4002 Debrecen, Hungary}
\email{gyory@science.unideb.hu, apinter@science.unideb.hu}

\address{Vandita Patel\newline
\indent Alan Turing Building, Department of Mathematics,\newline
\indent The University of Manchester,\newline 
\indent Oxford Road, Manchester, M13 9PL, UK}
\email{vandita.patel@manchester.ac.uk}

\address{Samir Siksek\newline
\indent Mathematics Institute, Zeeman Building \newline
\indent University of Warwick, \newline
\indent Coventry, CV4 7AL, UK}
\email{samir.siksek@gmail.com}

\thanks{The research was supported in part by the NKFIH grants ANN130909 and K128088 of the Hungarian National Research, Development and Innovation Office.
}

\date{}

\begin{abstract}
Let $f$ be a positive integer, and let $\chi$ be a primitive quadratic character of conductor $f$. 
	Let $k$ be a positive
	integer, and write $B_k(\chi,X)$ for the $k$-th Bernoulli polynomial
	corresponding to $\chi$. Suppose $B_k(\chi,X)$ is irreducible and of
	degree at least $2$.
	 Then for 100\% of positive
integers $m$ divisible by $f$, the Diophantine equation
\[
	\chi(1) \cdot (x+1)^k+\chi(2) \cdot (x+2)^k+\cdots+\chi(m) \cdot (x+m)^k \, =\, y^n,
\]
has no solutions with $x$, $y$, $n$ integers, and $n \ge 2$.
\end{abstract}

\maketitle

\section{Introduction}

In 1956, Sch\"{a}ffer \cite{Schaffer}, inspired by the Erd\H{o}s--Moser conjecture, considered the now classic equation
\begin{equation}\label{eqn:Schaffer}
	1^k+2^k+\cdots+x^k=y^n.
\end{equation}
He gave four parametric families of solutions, and made a conjecture describing all solutions with $n \ge 2$ and $k \ge 1$. For a survey of results related to Sch\"{a}ffer's conjecture, see \cite{Survey}.
Dilcher \cite{Dilcher86}  considered a variant of \eqref{eqn:Schaffer}, 
proving the following theorem.  
\begin{theorem}[Dilcher, 1986]\label{thm:Dilcher86}
Let $f$ be a positive integer,
and let $\chi$ be a primitive quadratic character
of conductor $f$. Let $b \ne 0$. 
If $k$ is a sufficiently large fixed integer, then the equation
	\begin{equation}\label{eqn:Dilcher}
		\chi(1) \cdot 1^k+\chi(2) \cdot 2^k+\cdots+\chi(xf) \cdot (xf)^k \, = \, b y^n
	\end{equation}
has finitely many solutions with $x$, $y \ge 1$ and $n \ge 2$.
\end{theorem}
Bennett \cite{twistedBennett} refers to \eqref{eqn:Dilcher} as a \lq\lq character-twisted\rq\rq\ version of Sch\"{a}ffer's classic equation,
and determines all solutions with $b=-1$ and $3 \le k \le 6$, for the unique non-trivial primitive quadratic
character of conductor $4$.

\bigskip

In this paper we consider a shifted version of Dilcher's character-twisted equation \eqref{eqn:Dilcher}.
\begin{theorem}\label{thm:main}
Let $b \ne 0$, $k \ge 2$ be integers.  
Let $f$ be a positive integer, and let $\chi$ be a primitive quadratic
character of conductor $f$.
	Write $B_k(\chi,X)$ for the $k$-th Bernoulli polynomial corresponding to the character $\chi$ (we introduce this below).
	Suppose $B_k(\chi,X)$ is irreducible of degree $\ge 2$.  Then for 100\% of positive
integers $m$ satisfying $f \mid m$, we have that there are no solutions to the equation
\begin{equation}\label{equation:powsum}
\sum_{0 \le d<m} \chi(d) \cdot (x+d)^k \, =\, b y^n, \qquad x,~y \in \mathbb{Z}, \quad n \ge 2.
\end{equation}
\end{theorem}

Very little is known regarding irreducibility of generalized Bernoulli polynomials;
the following theorem of Dilcher \cite{Dilcher87} gives a sufficient, though far from necessary, condition.
\begin{theorem}[Dilcher]
	Let $p \equiv 3 \pmod{8}$ be a prime. Let 
	$\chi$ be an odd (i.e.\ $\chi(-1)=-1)$, non-principal quadratic character with conductor $p$.
        Then $B_k(\chi,X)$ is irreducible for odd $k\geq 1$.
\end{theorem}

The proof of Theorem~\ref{thm:main} uses ideas from a theorem of Patel and Siksek \cite{PatelSiksek},
where the following theorem is established.
\begin{theorem}
Let $k \ge 2$ be even.
Then for $100\%$ of positive integers $m$, we have that there are no solutions to the equation
\[
	\sum_{0 \le d<m} (x+d)^k \, =\, y^n, \qquad x,~y \in \mathbb{Z}, \quad n \ge 2.
\]
\end{theorem}

\section{Generalised Bernoulli Numbers and Polynomials}
Our main reference for this section is Chapter 9.4 of Cohen's textbook \cite{Cohen}.
Let $\chi$ be a primitive Dirichlet character with conductor $f$.

\begin{defn}
	For $k \ge 0$, we define the $B_k(\chi,X)$ via the generating series
	\begin{equation}\label{eqn:generating}
te^{tX}\frac{\sum_{0\leq a<f} \chi(a)e^{at}}{e^{ft} -1} \, = \, \sum_{k \geq 0} \frac{B_k(\chi,X)}{k!}t^k.
	\end{equation}
We call $B_k(\chi,X)$ the $k$-th Bernoulli polynomial corresponding to $\chi$.
\end{defn}

We shall make use of the following property \cite[Proposition 9.4.4]{Cohen} 
relating the derivative of $B_k(\chi,X)$ to $B_{k-1}(\chi,X)$:
	\begin{equation}\label{eqn:derivative}
B^{\prime}_{k}(\chi,X) = kB_{k-1}(\chi,X).
	\end{equation}

Let $m$ be a positive integer and suppose $f \mid m$. We make use of the following relation 
\cite[Proposition 9.4.8]{Cohen} between $\chi$-twisted power sums and
generalised Bernoulli polynomials.
\begin{equation}\label{equation:sum}
    \sum_{0 \leq d <m} \chi(d)(X+d)^{k} = \frac{1}{k+1} \left( B_{k+1}(\chi,X + m) - B_{k+1}(\chi,X) \right).
\end{equation}

\section{Counting Integers with Restrictions on their Prime Factorisation}
Let $\PP$ be the set of prime numbers
and let $S \subseteq \PP$. Following Serre's paper \cite{Serre},
we say that $S$ has \textbf{regular density $\alpha>0$}
if, for $\re(s)>1$,
\begin{equation}\label{eqn:regular}
\sum_{p \in S} \frac{1}{p^s} \; =\; \alpha \cdot \log\left( \frac{1}{s-1} \right) \; + \; \theta(s)
\end{equation}
where $\theta$ extends to a holomorphic function on $\re(s) \ge 1$.
We say that the set $S$ has \textbf{Frobenian density $\alpha>0$}
if there exists a finite Galois extension $K/\Q$
and a subset $C$ of $G=\Gal(K/\Q)$, such that
\begin{itemize}
\item $C$ is a union of conjugacy classes in $G$;
\item $\alpha=\# C/\# G$;
\item for every sufficiently large prime $p$,
we have $p \in S$ if and only if $\Frob_p \in C$
where $\Frob_p$ is a Frobenius element of $G$
corresponding to $p$.
\end{itemize}
By the Chebotarev Density Theorem (e.g. \cite[Proposition 1.5]{Serre}),
 if $S$ has Frobenian density $\alpha>0$
then it has regular density $\alpha>0$.
We shall need the following standard sort of result from
analytic number theory. We cannot find a suitable
reference, so we give a proof, which is in fact an adaptation of
the proof of similar theorem \cite[Theorem 10]{KS_stab}.
\begin{theorem}\label{thm:count}
        Let $S$ be a set of primes of regular density $0 < \alpha<1$.
        Let $N_S(X)$ be the set of positive integers $n \le X$
        such that $v_p(n) \ne 1$ for all $p \in S$.
        Then, there exists some constant $C_S>0$ such that
        \[
                \# N_S(X) \; \thicksim \; C_S \cdot \frac{X}{(\log{X})^\alpha}.
        \]
\end{theorem}
\begin{proof}
        Let
        \[
                a_n=\begin{cases}
                        1 & \text{if $v_p(n) \ne 1$ for all $p \in S$}\\
                        0 & \text{otherwise}.
                \end{cases}
        \]
        Then $\# N_S(X)=\sum_{n \le X} a_n$. We shall make use of a Tauberian theorem
        due to Delange \cite[page 350]{Tenenbaum} to estimate $\#N_S(X)$.

        Consider the Dirichlet series
        \[
                D(s)=\sum_{n=1}^\infty \frac{a_n}{n^s},
        \]
        which defines a holomorphic function on $\re(s)>1$.
        We need to analyse the behaviour of $D(s)$ in the neighbourhood
        of $s=1$.
        For $\re(s)>1$ we write $D(s)$ as an Euler product
        \[
                        D(s)  =\prod_{p \in S} \left(1+ \frac{1}{p^{2s}}+
			\frac{1}{p^{3s}}+\frac{1}{p^{4s}}+\cdots \right)
                \cdot \prod_{p \notin S} \left(1+ \frac{1}{p^{s}}+\frac{1}{p^{2s}}+\cdots \right) .
        \]
        Thus
        \[
                \log(D(s)) \; = \; \sum_{p \notin S} \frac{1}{p^s} \; + \; \theta(s)
        \]
        where $\theta$ is holomorphic at $s=1$. Since $\mathbb{P} \setminus S$
        has regular density $1-\alpha>0$, we conclude from \eqref{eqn:regular} that
        \[
                \log(D(s)) \; = \; (1-\alpha) \cdot \log \left(\frac{1}{s-1} \right)+\phi(s)    
        \]
        where $\phi$ is holomorphic at $s=1$. Hence
        \[
                D(s) \; = \; \frac{g_0(s)}{(s-1)^{1+\omega}}
        \]
        where $g_0(s)=\exp(\phi(s))$ is holomorphic and non-vanishing at $s=1$, and $\omega=-\alpha$;
        the notation here is chosen to match that of Theorem II.7.28 of \cite{Tenenbaum}.
        We note that $-1<\omega < 0$. Applying the aforementioned theorem gives
        \[
                \sum_{n \le X} a_n \; \thicksim \; 
                \frac{g_0(1)}{\Gamma(\omega+1)} \cdot X(\log{X})^\omega,
        \]
	where $\Gamma$ is the Gamma function.
        This gives the theorem with $C_S=\exp(\phi(1))/\Gamma(1-\alpha)$.
\end{proof}

\section{Proof of Theorem~\ref{thm:main}}
From now on $\chi$ will be a primitive quadratic character. Thus $\chi$ takes values
in $\{-1,0,1\}$, and it immediately follows from \eqref{eqn:generating} that $B_k(\chi,X) \in \Q[X]$ for all $k$.
\begin{lemma}\label{lem:Taylor}
Let $q$  be prime not dividing any of the denominators of 
	the coefficients appearing in $B_{k+1}(\chi,X)$. 
	Let $a$, $m$ be integers, with $m \geq 2$. Suppose $q \mid m$ and $f \mid m$. Then
\[
   \sum_{0 \leq d <m} \chi(d)(a+d)^{k} \equiv m B_k(\chi,a) \pmod{q^2}
\]
\end{lemma}
\begin{proof}
By Taylor's Theorem,
\[
B_{k+1}(\chi,X+m)=B_{k+1}(\chi,X)+m B^\prime_{k+1}(\chi,X) \pmod{q^2}.
\]
	But $B_{k+1}^\prime(\chi,X)=(k+1) B_k(\chi,X)$ by \eqref{eqn:derivative}. The lemma follows from \eqref{equation:sum}.
\end{proof}

\begin{lemma}\label{lem:main}
    Suppose $B_k(\chi,X)$ is irreducible of degree $\ge 2$. Then there is a set of primes
	$S$ having positive regular density, such that for any $a \in \mathbb{Z}$,
    \[
    v_q\left(\sum_{0 \le d <m} \chi(d)(a+d)^k\right)=1
    \]
    whenever $q \in S$ and $q \mid\mid m$ and $f \mid m$.
\end{lemma}
\begin{proof}
Let $G$ be the Galois group of $B_k(\chi,X)$, which we think of
as a transitive permutation group on the roots. Let $C$ be the set of elements of $G$
	that act freely on the roots. By a theorem of Jordan \cite{Jordan},
	the set $C$ is non-empty. Clearly $C$ is a union of
	conjugacy classes. We let $S$ be the 
	set of primes $q$ such that $\Frob_q \in C$. This set
	has positive Frobenian and hence positive regular density.
	We remove from $S$, without affecting the density,
	those primes that appear in the denominators
	of the coefficients of $B_k(\chi,X)$, or divide
	the numerator of its discriminant or its leading
	coefficient. It follows
	that if $q \in S$ then $B_k(\chi,X)$ has no roots
	modulo $q$. 
	Thus, for any $a \in \mathbb{Z}$, we have $v_q(B_k(\chi,a))=0$. The lemma follows from Lemma~\ref{lem:Taylor}.
\end{proof}

\begin{proof}[Proof of Theorem~\ref{thm:main}]
Let $S$ be as in Lemma~\ref{lem:main}.
We exclude from $S$ the primes that divide $b$; this does not affect the fact that $S$ has positive regular density.
Let $\cA$ be the set of integers $m$ with $f \mid m$ such that \eqref{equation:powsum} has a solution, and write $\cA(X)$ for the set 
of $m \le X$ belonging to $\cA$. 
Let $N_S(X)$ be as in Theorem~\ref{thm:count};
by that theorem it will be sufficient to show that $\cA(X) \subseteq N_S(X)$.
	Let $m \in \cA(X)$. Then we have a solution \eqref{equation:powsum},
	say $x=a$, and $y=c$. Thus,
	\[
		\sum_{0 \le d<m} \chi(d)(a+d)^k=b c^n
	\]
	with $n \ge 2$. Let $q \in S$. Then $v_q(b c^n) \ne 1$.
	By Lemma~\ref{lem:main}, $v_q(m) \ne 1$.
	Thus $m \in N_S(X)$ as required.
\end{proof}

\bibliographystyle{abbrv}
\bibliography{Unit}

@incollection {Survey,
    AUTHOR = {Coppola, Nirvana and Curc\'o-Iranzo, Mar and Khawaja, Maleeha
              and Patel, Vandita and \"Ulkem, \"Ozge},
     TITLE = {Power values of power sums: a survey},
 BOOKTITLE = {Women in numbers {E}urope {IV}---research directions in number
              theory},
    SERIES = {Assoc. Women Math. Ser.},
    VOLUME = {32},
     PAGES = {155--193},
 PUBLISHER = {Springer, Cham},
      YEAR = {[2024] \copyright 2024},
      ISBN = {978-3-031-52162-1; 978-3-031-52163-8},
   MRCLASS = {11D41},
  MRNUMBER = {4786493},
MRREVIEWER = {Pedro-Jos\'e\ Cazorla Garc\'ia},
       DOI = {10.1007/978-3-031-52163-8\_6},
       URL = {https://doi.org/10.1007/978-3-031-52163-8_6},
}

@article {Schaffer,
    AUTHOR = {Sch\"affer, Juan J.},
     TITLE = {The equation {$1^p+2^p+3^p+\cdots+n^p=m^q$}},
   JOURNAL = {Acta Math.},
  FJOURNAL = {Acta Mathematica},
    VOLUME = {95},
      YEAR = {1956},
     PAGES = {155--189},
      ISSN = {0001-5962,1871-2509},
   MRCLASS = {10.1X},
  MRNUMBER = {78395},
MRREVIEWER = {J.\ W. S. Cassels},
       DOI = {10.1007/BF02401100},
       URL = {https://doi.org/10.1007/BF02401100},
}

@book {Cohen,
    AUTHOR = {Cohen, Henri},
     TITLE = {Number theory. {V}ol. {II}. {A}nalytic and modern tools},
    SERIES = {Graduate Texts in Mathematics},
    VOLUME = {240},
 PUBLISHER = {Springer, New York},
      YEAR = {2007},
     PAGES = {xxiv+596},
      ISBN = {978-0-387-49893-5},
   MRCLASS = {11-01 (11D61 11F80 11J86 11Mxx)},
  MRNUMBER = {2312338},
MRREVIEWER = {R.\ C.\ Baker},
}

@article {PatelSiksek,
    AUTHOR = {Patel, Vandita and Siksek, Samir},
     TITLE = {On powers that are sums of consecutive like powers},
   JOURNAL = {Res. Number Theory},
  FJOURNAL = {Research in Number Theory},
    VOLUME = {3},
      YEAR = {2017},
     PAGES = {Paper No. 2, 7},
      ISSN = {2522-0160,2363-9555},
   MRCLASS = {11D41 (11B68 11D61 11F80)},
  MRNUMBER = {3608381},
MRREVIEWER = {Reese\ Scott and Robert\ Styer},
       DOI = {10.1007/s40993-016-0068-0},
       URL = {https://doi.org/10.1007/s40993-016-0068-0},
}

@article {Dilcher87,
    AUTHOR = {Dilcher, Karl},
     TITLE = {Irreducibility of certain generalized {B}ernoulli polynomials
              belonging to quadratic residue characters},
   JOURNAL = {J. Number Theory},
  FJOURNAL = {Journal of Number Theory},
    VOLUME = {25},
      YEAR = {1987},
    NUMBER = {1},
     PAGES = {72--80},
      ISSN = {0022-314X,1096-1658},
   MRCLASS = {11B68 (11R09)},
  MRNUMBER = {871169},
MRREVIEWER = {T.\ Mets\"ankyl\"a},
       DOI = {10.1016/0022-314X(87)90016-3},
       URL = {https://doi.org/10.1016/0022-314X(87)90016-3},
}

@article {Dilcher86,
    AUTHOR = {Dilcher, Karl},
     TITLE = {On a {D}iophantine equation involving quadratic characters},
   JOURNAL = {Compositio Math.},
  FJOURNAL = {Compositio Mathematica},
    VOLUME = {57},
      YEAR = {1986},
    NUMBER = {3},
     PAGES = {383--403},
      ISSN = {0010-437X,1570-5846},
   MRCLASS = {11D61 (11B68)},
  MRNUMBER = {829328},
MRREVIEWER = {T.\ Mets\"ankyl\"a},
       URL = {http://www.numdam.org/item?id=CM_1986__57_3_383_0},
}

@article {twistedBennett,
    AUTHOR = {Bennett, Michael A.},
     TITLE = {A superelliptic equation involving alternating sums of powers},
   JOURNAL = {Publ. Math. Debrecen},
  FJOURNAL = {Publicationes Mathematicae Debrecen},
    VOLUME = {79},
      YEAR = {2011},
    NUMBER = {3-4},
     PAGES = {317--324},
      ISSN = {0033-3883,2064-2849},
   MRCLASS = {11D41 (11B68)},
  MRNUMBER = {2907968},
MRREVIEWER = {P.\ Bundschuh},
       DOI = {10.5486/PMD.2011.5081},
       URL = {https://doi.org/10.5486/PMD.2011.5081},
}

@misc{KS_stab,
      title={New Algebraic Points on Curves}, 
      author={Maleeha Khawaja and Samir Siksek},
      year={2025},
      eprint={2511.15635},
      archivePrefix={arXiv},
      primaryClass={math.NT},
      url={https://arxiv.org/abs/2511.15635}, 
      note={https://arxiv.org/abs/2511.15635},
}

@article{Jordan,
     author = {Jordan, Camille},
     title = {Recherches sur les substitutions},
     journal = {Journal de Math\'ematiques Pures et Appliqu\'ees},
     pages = {351--367},
     publisher = {Gauthier-Villars},
     volume = {2e s{\'e}rie, 17},
     year = {1872},
     language = {fr},
     url = {https://www.numdam.org/item/JMPA_1872_2_17__351_0/}
}

@book {Tenenbaum,
    AUTHOR = {Tenenbaum, G\'{e}rald},
     TITLE = {Introduction to analytic and probabilistic number theory},
    SERIES = {Graduate Studies in Mathematics},
    VOLUME = {163},
   EDITION = {Third},
      NOTE = {Translated from the 2008 French edition by Patrick D. F. Ion},
 PUBLISHER = {American Mathematical Society, Providence, RI},
      YEAR = {2015},
     PAGES = {xxiv+629},
      ISBN = {978-0-8218-9854-3},
   MRCLASS = {11-02 (11Kxx 11Mxx 11Nxx)},
  MRNUMBER = {3363366},
       DOI = {10.1090/gsm/163},
}

@incollection {Serre,
    AUTHOR = {Serre, Jean-Pierre},
     TITLE = {Divisibilit\'{e} de certaines fonctions arithm\'{e}tiques},
 BOOKTITLE = {S\'{e}minaire {D}elange-{P}isot-{P}oitou, 16e ann\'{e}e (1974/75),
              {T}h\'{e}orie des nombres, {F}asc. 1, {E}xp. {N}o. 20},
     PAGES = {28},
      YEAR = {1975},
   MRCLASS = {10D15},
  MRNUMBER = {0392831},
MRREVIEWER = {E. Grosswald},
}
\end{document}